\documentclass[11pt]{article}

\usepackage{subcaption}
\usepackage[T1]{fontenc}
\usepackage[utf8]{inputenc}
\usepackage[british]{babel}
\usepackage[
rm={oldstyle=false,proportional=true},
sf={oldstyle=false,proportional=true},
tt={oldstyle=false,proportional=true}
]{cfr-lm}
\usepackage[babel]{microtype}

\usepackage{color,xcolor}

\usepackage{authblk}

\usepackage{mathrsfs}
\usepackage{mathtools}
\usepackage{amsfonts}
\usepackage{amssymb}
\usepackage{amsthm}

\usepackage{enumerate}

\usepackage[misc]{ifsym}

\usepackage[noadjust]{cite}
\usepackage[hidelinks,bookmarksnumbered]{hyperref}
\usepackage[capitalise,noabbrev]{cleveref}

\usepackage{orcidlink}

\newtheorem{theorem}{Theorem}
\newtheorem{lemma}[theorem]{Lemma}
\newtheorem{proposition}[theorem]{Proposition}
\newtheorem{corollary}[theorem]{Corollary}

\theoremstyle{definition}
\newtheorem{definition}[theorem]{Definition}
\newtheorem{conjecture}[theorem]{Conjecture}

\newtheorem{problem}[theorem]{Open problem}

\DeclareMathOperator{\cop}{c}
\DeclareMathOperator{\diam}{diam}

\DeclareMathOperator{\ded}{\iota}

\DeclareMathOperator{\evcapt}{ec}
\DeclareMathOperator{\reach}{r}

\renewcommand{\geq}{\geqslant}
\renewcommand{\ge}{\geqslant}
\renewcommand{\leq}{\leqslant}
\renewcommand{\le}{\leqslant}

\newcounter{authcount}

\NewDocumentCommand{\authDetails}{m m m o}{%
    \stepcounter{authcount}%
    \IfNoValueTF{#4}{%
        \author[\arabic{authcount}]%
    }%
    {%
        \author[#4]%
    }%
    {%
        \mbox{#1\,$^{\textrm{\href{mailto:#2}{\Letter}}\,\,%
        \ifx&#3&\else\raisebox{-0.2ex}{\orcidlink{#3}}\,\fi}$}%
    }%
}

\title{Seeing is not believing in limited visibility cops and robbers}

\authDetails{Bojan Ba\v{s}i\'c\footnote{Supported by the Ministry of Science,
Technological Development and Innovation of the Republic of Serbia (grants no.
451-03-137/2025-03/200125 and
451-03-136/2025-03/200125).}}{bojan.basic@dmi.uns.ac.rs}{}[1]
\authDetails{Alfie Davies}{research@alfied.xyz}{0000-0002-4215-7343}[2]
\authDetails{Aleksa D\v{z}uklevski}{aleksa.dzuklevski@dmi.uns.ac.rs}{}[1]
\authDetails{Strahinja Gvozdi\'c}{sgvozdic@ethz.ch}{}[3]
\authDetails{Yannick Mogge\footnote{This research was partly supported by the
ANR project P-GASE
(ANR-21-CE48-0001-01).}}{yannick.mogge@univ-lyon1.fr}{0000-0003-4239-9112}[4]

\affil[1]{ Department of Mathematics and Informatics\\University of Novi Sad\\Serbia}
\affil[2]{ Memorial University of Newfoundland\\Canada}
\affil[3]{ ETH Zurich\\Switzerland}
\affil[4]{ Université Claude Bernard Lyon 1\\France}

\date{}

\begin{document}
\maketitle

\begin{abstract}
    \noindent
    We consider the model of limited visibility Cops and Robbers, where the
    cops can only see within their $l$-neighbourhood. We prove that the number
    of cops needed to see the robber can be arbitrarily smaller than the number
    needed to capture the robber, answering an open question from the
    literature. We then consider how close we can get to seeing the robber when
    we do not have enough cops, along with a probabilistic interpretation.
\end{abstract}

\section{Introduction}

In the game of cops and robbers on graphs, one player controls a team of cops
and attempts to catch a robber (controlled by the other player) while moving
along the edges of a graph. After the cops and robbers pick their starting
vertices (in that order), the players take turns: the player controlling the
cops decides whether each cop remains stationary or else moves to an adjacent
vertex; the other player then moves the robber by the same rules. Note that
multiple cops can occupy the same vertex. At every moment in the game, both
players have complete knowledge of the graph and the positions of the pieces.
The game ends when a cop occupies the same vertex as the robber, thereby
capturing him. If the robber can evade capture indefinitely, then the game
never ends, and we say that the robber wins.

Similar games on graphs, where some pursuers aim to catch an evader, have been
around for at least half a century (an early reference is
\cite{parsons:pursuit-evasion}), while their non-discrete counterparts have
been in existence for even longer (e.g.\, the monograph
\cite{isaacs:differential}). The particular game from the preceding paragraph
first appeared in Quilliot's thesis \cite{quilliot:jeux}. This predates
Nowakowski and Winkler's work \cite{nowakowski.winkler:vertex-to-vertex}, which
independently considered the same game, although the latter is still sometimes
referred to as a starting point of the literature (as mentioned in
\cite{bonato.nowakowski:game}). In both works
(\cite{quilliot:jeux,nowakowski.winkler:vertex-to-vertex}), only the version
with one cop is studied (and all graphs for which the cop wins are
characterised). The game with more cops was considered for the first time by
Aigner and Fromme \cite{aigner.fromme:game}; this initiated the study of the
\emph{cop number} of a graph, which eventually became a central topic in the
research of vertex-pursuit games on graphs. The \emph{cop number} of a graph
$G$, denoted $\cop(G)$, is the minimum number of cops required to guarantee the
capture of the robber, regardless of the robber's strategy.

The literature concerning the cop numbers of graphs is now too vast to be
surveyed in a short introduction here; see Bonato and Nowakowski's monograph
\cite{bonato.nowakowski:game} for a good background on the topic. Even with its
breadth, some generally applicable theorems are rather sparse. For example, it
is unknown what the growth rate of the maximal cop number among all connected
graphs with $n$ vertices (for $n\to\infty$) is. It was conjectured by Meyniel
in 1985 to be $O(\sqrt{n})$ (this conjecture was a personal communication to
Frankl, as mentioned in \cite{frankl:cops} in 1987).

Many researchers would consider Meyniel's conjecture to be the main question in
the area, but, regarding the state of the art, it is not even known whether
there exists a single positive $\varepsilon$ for which the bound
$O(n^{1-\varepsilon})$ holds (this is sometimes referred to as \emph{weak
Meyniel's conjecture}); for the best known upper bounds, an interested reader
may check \cite{scott.sudakov:bound} or \cite{lu.peng:on}. It is known,
however, that if the bound from Meyniel's conjecture is true, then it is the
best possible. Indeed, already in \cite{aigner.fromme:game} it was proven that,
if the girth of $G$ is at least $5$, then $\cop(G)$ is no less than the minimum
degree of $G$; as incidence graphs of finite projective planes have girth $6$
and minimum degree $\Omega(\sqrt n)$, the conclusion follows. (It is possible
that this particular construction appeared for the first time in print only in
\cite{prałat:when}, but, in \cite{frankl:cops}, where Meyniel's conjecture is
originally introduced, it is immediately added: ``\dots which would be the best
possible,'' without implying this nor any other construction.) It will turn out
that this construction is relevant also for our present work, later in
\cref{subsec:girth}.

Given the popularity of this game (and perhaps also the difficulty of more
general results), various variations have appeared, such as: when the cops are
lazy \cite{bal.bonato.ea:lazy}; when there are zombies and survivors instead of
cops and robbers \cite{fitzpatrick.howell.ea:deterministic}; a combination of
the two, with lazy zombies \cite{bose.de-carufel.ea:pursuit-evasion}; when the
robber is invisible and possibly drunk \cite{kehagias.mitsche.ea:cops}; when
there is a cheating robot instead of a robber \cite{huggan.nowakowski:cops};
where the cops are using helicopters against an infinitely fast (and possibly
invisible) robber~\cite{seymour.thomas:graph} etc. These colourful examples are
to name but a few; see \cite[Chapters 8 and 9]{bonato.nowakowski:game} for
more. Regarding the application of (some of) those variants, there is a rather
audacious (visionary?) claim in \cite{bonato:what}: ``\dots and variants of the
game have been recently considered in fields such as robotics and mathematical
counter-terrorism'' (perhaps referring to yet another vividly-named version of
the game, called Seepage; the introduction of
\cite{bonato.mitsche.ea:vertex-pursuit} elaborates more on this).

Yet another direction in which the problem can be generalised is if the current
distance between the cops and the robber matters: for example, if the cops'
goal is just to reach within a distance of $k$ from the robber (where $k$ is a
parameter given in advance)---that is (if one is willing to accept a perhaps
overly graphic metaphor), to reach a distance from which they could shoot the
robber \cite{bonato.chiniforooshan:pursuit,chalopin.chepoi.ea:cop}. But what is
possibly more studied is the variant of `limited visibility', in which one of
the players cannot see the other unless the distance between them is at most
$l$ (where $l$ is a parameter given in advance).

One special case of this, which appears to have received the most attention in
the literature, is when the cops cannot see the robber at all; one might
imagine that he is invisible (or the cops are blind, and can detect the robber
only by bumping into him---in other words, occupying the same vertex as the
robber). We have already presented some references that treat an invisible
robber in various settings (and for more, see \cite{bienstock:graph,
alspach:searching, fomin.thilikos:annotated, bonato.yang:graph}).
Chronologically, however, the first reference that describes the rules of
classical cops and robbers with the addition of robber invisibility appears to
be \cite{tosic:vertex-to-vertex}. Making the robber invisible indeed makes
sense from a practical point of view, since this models the problem where a
person has to be \emph{found}, not only pursued (like a prince searching for
the princess \cite{britnell.wildon:finding}). In that vein, it is of note that
the reference \cite{parsons:pursuit-evasion} (which we cited earlier as a
genesis for a lot of research) was first inspired by an article published in a
spelunkers' journal \cite{breisch:intuitive}!

Other settings of limited visibility for one or both sides (and not only with
$l=0$), have been studied in \cite{aleliunas.karp.ea:random,
adler.racke.ea:randomized, isler.kannan.ea:randomized, isler.karnad:role,
abramovskaya.fomin.ea:how}. It is a curiosity that, in the very article where
the cop number originated \cite{aigner.fromme:game}, the authors write: ``One
could for example, allow complete information only when $C$ and $R$ are at most
a distance $d$ apart ($C$ and $R$ have `eye-contact'),'' without further
explanation. It took a while for this idea to catch on.

The present research has been provoked in particular by the article
\cite{clarke.cox.ea:limited}. When cops have limited visibility $l$ (but the
robber has no such restriction), we can consider a game with the classical goal
(they want to capture the robber), but also a game in which their goal is just
to see the robber (essentially a combination of the classical game with limited
visibility, along with the trigger-happy cops from a few paragraphs ago). In
\cite{clarke.cox.ea:limited}, it was posed as an open question whether the
difference of those two versions of the cop number (where $l$ is fixed) can be
arbitrarily large; that is, can the cost of wanting to capture the robber
instead of seeing them be arbitrarily high? In the present article, we answer
this question in the affirmative (\cref{sec:cost}).

In \cref{sec:cleaning}, we study the game from an alternative viewpoint akin to
edge-searching, where the cops are trying to clean as many vertices in the
graph as possible; that is, the cops are trying to limit the possible locations
of the robber to as small a set as they can. In particular, we discuss: the
difference between being able to \emph{see} the robber versus being able to
\emph{infer} their location (because of where the cops know the robber is not);
how many vertices (in more generality), the cops are able to guarantee do not
contain the robber; which graphs are particularly difficult to clean well.

In \cref{sec:prob}, we discuss a probabilistic interpretation of the game;
usually, in limited visibility cops and robbers, the robber is assumed to be
omniscient, having knowledge of everything that the cops will do in the future,
and so they will never `accidentally' stumble into the cops. But what happens
when the robber is not omniscient? How many cops do we need to eventually
capture (or see) the robber? It does not take much work to show, for example,
that wanting to capture the robber with $l$-visibility cops requires exactly
the same number of classical cops to capture the robber. It may be worthwhile
investigating differences in (expected) capture time, however.

\section{The cost of capture}
\label{sec:cost}

In light of considering cops with limited visibility, let us start this section
by recalling the following definitions.

\begin{definition}
    If $G$ is a graph and $l$ is a non-negative integer, then we define the
    \emph{$l$-visibility cop number} of $G$, written $\cop_l(G)$, as the
    minimum number of cops with visibility $l$ required to guarantee seeing the
    robber, regardless of the robber's strategy (but assuming an optimal cop
    strategy).

    Similarly, we define the \emph{$l$-visibility capture number} of $G$,
    written $\cop'_l(G)$, as the minimum number of cops with visibility $l$
    required to guarantee the capture of the robber.
\end{definition}

What we want to better understand are the relationships between these
parameters $\cop'_l(G)$ and $\cop_l(G)$, and also how they relate to the
classical $\cop(G)$. The following inequalities are immediate:
\[
    \cop'_l(G)\leq\cop_l(G),\quad\cop(G)\leq\cop_l(G).
\]

In \cite{clarke.cox.ea:limited}, a relationship was established between
$\cop_l(G)-\cop'_l(G)$ and $\cop_l(G)-\cop(G)$, as illustrated in their
theorem.

\begin{theorem}[{\cite[Theorem 3.1]{clarke.cox.ea:limited}}]
    If $G$ is a graph and $l\geq2$, then either $\cop'_l(G)=\cop_l(G)$ or else
    $\cop(G)\leq\cop_l(G)\leq\cop(G)+1$.
\end{theorem}

Furthermore, although they did not prove it specifically, it was written in
\cite{clarke.cox.ea:limited} that the difference $\cop_l(G)-\cop(G)$ can be
arbitrarily large (this followed swiftly from their Theorem 4.1).

\begin{theorem}[\cite{clarke.cox.ea:limited}]
    If $k$ and $l$ are positive integers, then there exists a graph $G$ with
    $\cop_l(G)-\cop(G)\geq k$.
\end{theorem}

It was then posed as an open problem whether there exists a graph $G$ such that
the difference between $\cop_l(G)-\cop'_l(G)$ can be arbitrarily large; i.e.\
what is the \emph{cost} of wanting to capture instead of to see?

\begin{problem}[c.f.\ {\cite[Problem 5.3]{clarke.cox.ea:limited}}]
    \label{prob:diff}
    If $k$ and $l$ are positive integers, does there exist a graph $G$ with
    $\cop_l(G)-\cop'_l(G)\geq k$?
\end{problem}

In this section, we answer \cref{prob:diff} in the affirmative: we show that
the difference can indeed be arbitrarily large. We split the proof into two
lemmas that consider the cases $l=1$ and $l>1$ separately.

In \cite[Corollary 3.2 on p.~408]{bonato.burgess:cops}, Bonato and Burgess
proved that for every non-negative integer $k$ there exists a graph $G$ of
diameter 2 such that $\cop(G)\geqslant  k$.

\begin{lemma}
    \label{lem:l>1}
    If $l>1$ and $k$ are positive integers, then there exists a graph $G$ such
    that $\cop_l(G)-\cop'_l(G)\geqslant k$.
\end{lemma}

\begin{proof}
    Let $G$ be a graph of diameter $2$ such that $\cop(G)\geqslant k+1$. Then,
    since $l>1$, observe that $\cop'_l(G)=1$ (the cop can be placed anywhere on
    the graph). But $\cop_l(G)\geq\cop(G)\geqslant  k+1$, and hence
    $\cop_l(G)-\cop'_l(G)\geqslant  k$.
\end{proof}

Of course, in the proof of \cref{lem:l>1}, the visibility of our cops was at
least as large as the diameter of the graph, and so seeing the robber was
trivial. But it is in fact not difficult to construct a graph with an
arbitrarily large diameter in which $\cop_l(G)-\cop_l'(G)\geqslant k$ still
holds for given $k$ and $l$: for any integer $d>2$, take the graph $G$ as in
the proof of \cref{lem:l>1}, add a new vertex $u$, and connect $u$ to any
vertex $v$ in $G$ by a path of length $d-1$. Let us call this new graph $G'$.
Clearly $\diam G\ge d$ and $\cop_l(G')\ge \cop_l(G)\ge k+1$.

Moreover, we still need only one cop to see the robber: let this cop start at
$u$ and move towards $v$ in the first $d-1$ moves. Once the cop has reached
$v$, the $u$--$v$ path will be clear of the robber, who must then be at one of
the vertices in $G$. But the cop can see all of $G$ from $v$, and so they can
see the robber as well. This gives $\cop_l'(G')=1$, and hence
$\cop_l(G')-\cop_l'(G')\ge k$.

While the case of $l>1$ was rather trivial (\cref{lem:l>1}), the case $l=1$
will be much more involved.

\begin{lemma}
    \label{lem:l=1}
    If $k$ is a positive integer, then there exists a graph $G$ such that
    $\cop_1(G)-\cop'_1(G)\geqslant k$.
\end{lemma}

\begin{proof}
    If $k=1$, then taking $G=C_4$ will suffice. Otherwise, we will construct
    $G$ in the following way:
    \begin{enumerate}
        \item
            For each integer $i$ with $0\le i\le 2k-1$, let $C_i$ be an ordered
            set of $2^{4k^2}$ elements and write $a_i$ for the $a^{\text{th}}$
            element of $C_i$. Add a vertex for each $a_i$, as $a$ and $i$ range
            over $\{0,1,\dots,2^{4k^2}-1\}$ and $\{0,1,\dots,2k-1\}$
            respectively.
        \item
            Associate each $C_i$ with the set $\{2^q:q\equiv i\pmod{2k}\text{
            and }0\leqslant q\leqslant4k^2-1\}$ of \emph{directions} (we in
            fact just need to take a regular $2k$-partition of those powers of
            2 such that no two of $2^q$, $2^{q+1}$, and $2^{q+2}$ are in the
            same $2k$-set).
        \item
            Add an edge between $a_i$ and $(a+2^q)_{i'}$ for each $i'\neq i$
            and each direction $2^q$ associated to $C_i$ (where the addition
            $a+2^d$, here and below, is computed modulo $2^{4k^2}$); such an
            edge will be called a \emph{forward} edge of $a_i$ of \emph{type}
            $2^q$ and a \emph{backward} edge of $(a+2^d)_{i'}$, again of type
            $2^q$ (but note that these are undirected edges). See \cref{fig:1}.
            \begin{figure}[!ht]
                \centering
                \includegraphics[width=12cm]{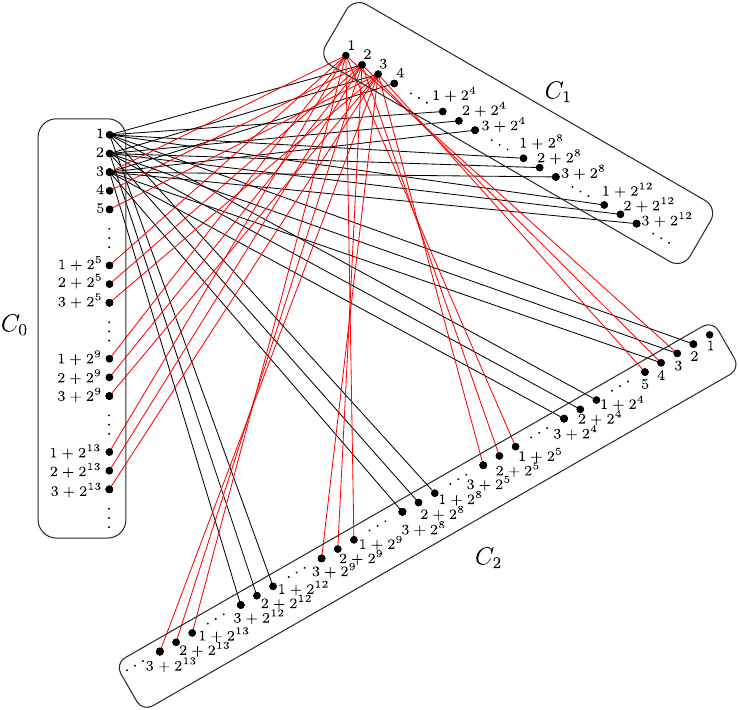}
                \caption{
                    Forward edges from the vertices $1$, $2$, and $3$ of $C_0$
                    that are going to $C_1$ and $C_2$ are shown here in black.
                    Forward edges from the vertices $1$, $2$, and $3$ of $C_1$
                    going to $C_0$ and $C_2$ are shown in red. The other edges
                    and components are omitted for clarity.
                }
                \label{fig:1}
            \end{figure}
        \item
            Add a $K_{2k}$, with vertices enumerated $v_1,\dots,v_{2k}$, and,
            for every $i$, add an edge between $v_i$ and each vertex in $C_i$.
            We will refer to the vertices in this $K_{2k}$ as the \emph{middle}
            of $G$, and the other vertices as the \emph{outside} of $G$. See
            \cref{fig:2}.
            \begin{figure}[!ht]
                \centering
                \includegraphics[width=10cm]{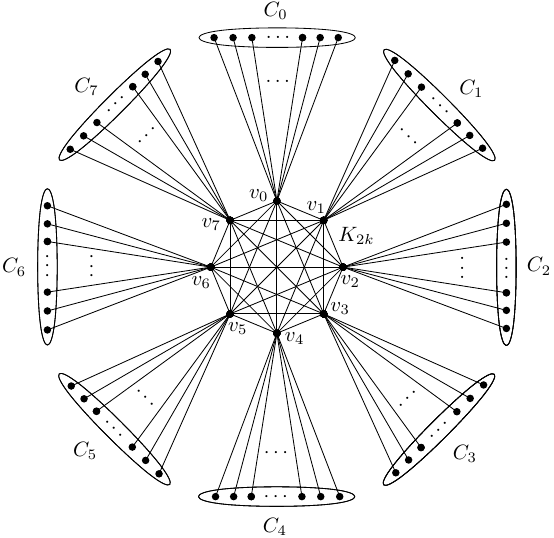}
                \caption{
                    All the edges of $G$ incident to the vertices of $K_{2k}$.
                    In the example shown here $k=4.$
                }
                \label{fig:2}
            \end{figure}
    \end{enumerate}

    Since the vertices in the middle of $G$ form a dominating set, it is clear
    that $\cop'_1(G)\leqslant   k$ simply by placing $k$ cops on distinct
    vertices in the middle of $G$ and then moving them to all of the remaining
    empty vertices in the middle on the first turn. Similarly, placing $2k$
    cops on distinct vertices in the middle will occupy the entire dominating
    set, and hence $\cop_1(G)\leq2k$. We will show now that $\cop_1(G)=2k$,
    which will yield a difference of $\cop_1(G)-\cop'_1(G)\geqslant k$, and
    hence also the result.

    We will consider when the robber is on an outside vertex $a_i$. Let a cop
    be on an outside vertex $b_j$ (where $a_i\neq b_j$). We claim that at most
    one type of forward edge of $a_i$ is incident with the closed neighbourhood
    of $b_j$. For the sake of contradiction, suppose this is not true and let
    $2^q$ and $2^{q'}$, $q\neq q'$, be distinct forward edges of $a_i$ whose
    other endpoints are both in the closed neighbourhood of $b_j$.

    Consider first when $b_j$ is itself the other endpoint of a forward edge of
    $a_i$. So,
    \[
        a+2^q\equiv b\pmod{2^{4k^2}}.
    \]
    We now distinguish two cases, based on whether the other forward edge from
    $a_i$ collides with:
    \begin{enumerate}
        \item[Case 1:] a forward edge of $b_j$

            We must have
            \begin{align*}
                &a+2^{q'}\equiv b+2^r\pmod{2^{4k^2}},\quad q\neq q'\\
                \implies\qquad&2^q\equiv2^{q'}-2^r\pmod{2^{4k^2}}\\
                \implies\qquad&2^q+2^r=2^{q'}\\
                \implies\qquad&q=r=q'-1,
            \end{align*}
            which contradicts $2^q$ and $2^{q+1}=2^{q'}$ not both being types
            of $C_i$.

        \item[Case 2:] a backward edge of $b_j$

            We must have
            \begin{align*}
                &a+2^{q'}\equiv b-2^t\pmod{2^{4k^2}},\quad q\neq q'\\
                \implies\qquad&2^q=2^{q'}+2^t\\
                \implies\qquad&q'=t=q-1,
            \end{align*}
            which contradicts $2^{q'}$ and $2^{q'+1}=2^q$ not both being types
            of $C_i$.
    \end{enumerate}
    So, if $b_j$ is the other endpoint of a forward edge of $a_i$, then it is
    incident with no other types of forward edges of $a_i$.

    Now suppose that $b_j$ is not the other endpoint of a forward edge of
    $a_i$. We again distinguish several cases based on the type of edges that
    connect $b_j$ to $a+2^q,a+2^{q'}$.
    \begin{enumerate}
        \item[Case 1:] $b_j$ is adjacent to $a+2^q,a+2^{q'}$ via two forward
            edges

            We must have
            \begin{align*}
                a+2^q&\equiv b+2^r\pmod{2^{4k^2}},\\
                a+2^{q'}&\equiv b+2^{r'}\pmod{2^{4k^2}},\quad r\neq r'.
            \end{align*}
            But it then follows that
            \begin{align*}
                &2^q-2^r\equiv2^{q'}-2^{r'}\pmod{2^{4k^2}}\\
                \implies\qquad&2^q+2^{r'}=2^{q'}+2^r\\
                \implies\qquad&q=r\text{ and }r'=q',
            \end{align*}
            which implies $i=j$, and so $a\neq b$ (since $a_i\neq b_j$),
            contradicting $a+2^q\equiv b+2^r\pmod{2^{4k^2}}$.
        \item[Case 2:] $b_j$ is adjacent to $a+2^q,a+2^{q'}$ via a forward edge
            and a backward edge

            We must have
            \begin{align*}
                a+2^q&\equiv b+2^r\pmod{2^{4k^2}},\\
                a+2^{q'}&\equiv b-2^{t}\pmod{2^{4k^2}},\quad r\neq t.
            \end{align*}
            But it then follows that
            \begin{align*}
                &2^q-2^r\equiv2^{q'}+2^{t}\pmod{2^{4k^2}}\\
                \implies\qquad&2^q=2^{q'}+2^t+2^r\\
                \implies\qquad&|t-r|=1,\, q'=\min(t,r),\, q=q'+2,
            \end{align*}
            where the last implication holds since $r\neq t$ (a forward and a
            backward edge of the same vertex cannot be of the same type by
            construction), which contradicts $2^{q'}$ and $2^{q'+2}=2^q$ not
            both being types of $C_i$.
        \item[Case 3:] $b_j$ is adjacent to $a+2^q,a+2^{q'}$ via two backward
            edges

            We must have
            \begin{align*}
                a+2^q&\equiv b-2^t\pmod{2^{4k^2}},\\
                a+2^{q'}&\equiv b-2^{t'}\pmod{2^{4k^2}},\quad t\neq t'.
            \end{align*}
            But it then follows that
            \begin{align*}
                &2^q+2^t=2^{q'}+2^{t'}\\
                \implies\qquad&q=t'\text{ and }t=q',
            \end{align*}
            which implies these two backward edges of $b_j$ are forward edges
            from $C_i$, contradicting the fact that the forward edges of $a_i$
            leave $C_i$.
    \end{enumerate}

    Thus, a cop on the outside of the graph can block at most one type of
    forward move of the robber (placed on the outside).

    Suppose now that we have $2k-1$ cops on the graph. If all $2k-1$ cops
    occupy distinct middle vertices and are not adjacent to the robber, then
    the robber may remain stationary and avoid capture for another turn.
    Otherwise, there must exist some $C_j$ that is not adjacent to a cop in the
    middle, where the robber is not in $C_j$. The robber has $2k$ types of
    forward moves to this $C_j$. Since there are no cops on $v_j$ (i.e.\ in the
    middle, adjacent to $C_j$), it follows that these $2k$ types need to be
    blocked by cops on the outside. But we have just shown that a cop on the
    outside can block at most one type of forward move of the robber.

    Thus, if there are fewer than $2k$ cops, then the robber has a winning
    strategy: staying on the outside, the robber will always have at least one
    move available to a vertex that is not in the neighbourhood of any cop. We
    saw earlier that $2k$ cops do indeed suffice to capture the robber, and
    hence we obtain $\cop_1(G)=2k$.
\end{proof}

Note that our construction builds a graph of diameter 3. However, we can use a
similar argument as before (with \cref{lem:l>1}) to increase the diameter
arbitrarily: simply add a vertex $u$ and a path of length $d-1$ from $u$ to the
vertex $v_0$ of the middle of $G$ to obtain $G'$. Then, $\diam G'\ge d$ and
$\cop_1(G')\ge\cop_1(G)$, and the strategy to show $\cop_1'(G') \leqslant k$ is
roughly the same as before: one cop starts from $u$ and moves towards $v_0$ in
the first $d-1$ turns, while the other cops remain stationary on vertices
$v_1,\dots,v_k$ in the middle of $G$. Then, in the $d^{\text{th}}$ turn, all
$k$ cops move to the remaining $k$ vertices of the middle, at which point they
guarantee seeing the robber.

Notice also that \cref{lem:l=1} gives us a graph such that
$\cop_1(G)\geq2\cop'_1(G)$. We believe that the ratio is exactly 2 for this
construction, but it is unclear whether it can be made arbitrarily large in
general. For the case $l>1$ (\cref{lem:l>1}), our construction did indeed allow
us to find graphs $G$ such that $\cop_l(G)\geqslant r\cop'_l(G)$ for
arbitrarily large $r$ (recall that we only needed one cop to see the robber).

\begin{problem}
    What is the largest $r$ for which there exists a graph $G$ with
    $\cop_1(G)\geqslant r\cop'_1(G)$?
\end{problem}

We know by our construction in \cref{lem:l=1} that the largest such $r$ must be
at least 2.

\begin{theorem}
    \label{thm:diff}
    If $l$ and $k$ are positive integers, then there exists a graph $G$ such
    that $\cop_l(G)-\cop'_l(G)\geqslant k$.
\end{theorem}

\begin{proof}
    This follows immediately from \cref{lem:l=1,lem:l>1}.
\end{proof}

This result was also answered independently and concurrently by Bill Kinnersley
and John Jones \cite{kinnersley:personal}, but their construction was markedly
different (making use of Hamming graphs with far fewer vertices, but obtaining
a smaller ratio than our 2 for the case $l=1$).

Note that, for all $l\ge1$, we have $\cop_l(G)\ge\cop(G)$. And so, in our
construction in \cref{lem:l=1}, we have $\cop_l(G)\ge2k$ for all $l\ge1$
(because we in fact proved that $\cop(G)=2k$, not just $\cop_1(G)=2k$). On the
other hand, increasing the visibility can only decrease the number of cops
required to see the robber: $\cop_l(G)\le\cop_1(G)\le k$. As such, for our $G$
from the proof of \cref{lem:l=1}, we obtain $\cop_l(G)-\cop_l'(G)\ge2k-k=k$ for
all $l\ge 1$. It is for this reason that our \cref{lem:l=1} essentially proves
the entire result (\cref{thm:diff}); the results of \cite{bonato.burgess:cops}
are not necessary, but it is interesting to compare the complexities of the two
proofs for the cases $l=1$ (\cref{lem:l=1}) and $l>1$ (\cref{lem:l>1}).

\section{Trying to clean a graph}
\label{sec:cleaning}

We will now examine the game from an alternative viewpoint, foregoing the
existence of a robber in the graph. Consider what the cops see with limited
visibility: before the game begins, cops are stationary at their respective
starting positions and, in general, can only see a limited, ``clean'' part of
the graph. Everything beyond that---including the robber's position---is
completely unknown. That is, the graph can be perceived as being covered in
fog, or gas.

As the cops move, the fog lifts as far as they can see. After the cops' turn,
the fog can move, but only slowly; after the cops' turn, the fog will spread to
all adjacent vertices not in sight of any cop. So, if there is a previously
visible vertex that, after the cops have moved, is no longer visible by any
cop, then it will become shrouded in fog now only if there is a neighbouring
vertex that is ``foggy'', or ``gaseous''.

Therefore, our game of limited visibility cops and robbers becomes a one-player
game with the goal of ``cleaning the graph of gas''. Indeed, this cleaning
viewpoint was also utilised in \cite{clarke.cox.ea:limited}. We formalise this
idea with the following definition.

\begin{definition}
    Given a graph $G$, the \emph{graph cleaning process} on $G$ with $k$
    cleaners and $l$-visibility proceeds as follows: $k$ cleaners are placed on
    the graph, and every vertex not in the closed $l$-neighbourhood of any of
    the cleaners is filled with gas; at each subsequent timestep, each cleaner
    is permitted to move to a vertex in its closed neighbourhood, cleaning all
    vertices in the new closed $l$-neighbourhood as well; after the cleaners
    have collectively ended their turn, every vertex not in the closed
    $l$-neighbourhood of the cleaners that is adjacent to a gaseous vertex will
    itself become gaseous (i.e.\ the gas will spread). (\cref{fig:C5} gives an
    example of one cop with 1-visibility trying to clean $C_5$.)
\end{definition}

\begin{figure}
    \begin{center}
        \includegraphics[width=5cm,page=10]{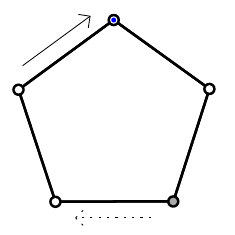}
        \caption{
            One cop with 1-visibility tries to clean a $C_5$. They can clean
            all but one vertex before the gas spreads back to the initial state
            of 2 gaseous vertices. (Blue vertices denote the vertices on which
            the cops are stationed, while the grey ones are the gaseous
            vertices, and the white ones are clean.) The picture shows the game
            state immediately after the cop moved (indicated by the normal
            arrow) and before the gas spreads (indicated by the dashed arrow).
        }
        \label{fig:C5}
    \end{center}
\end{figure}

It is immediate from the definition that $\cop_l'(G)$ is equal to the minimum
number of cleaners needed to clean $G$ of all gas (which explains why we will
still refer to cleaners as cops in some places). Following this point of view,
it is natural to ask how the number of cleaners would change if instead of
cleaning all vertices in $G$, one would require cleaning all \emph{but} at most
$r$ vertices for some $r>0$. In the original formulation, this would be the
same as asking how many possible locations we can narrow the robber to be in;
perhaps being able to say that there are only, say, two possible places the
robber can be is valuable information, even though we cannot locate them
precisely.

\begin{definition}
    \label{def:inference}
    Given a graph $G$, we define the $(r,l)$\emph{-inference number} of $G$,
    written $\ded^r_l(G)$, as the minimum number of cleaners with
    $l$-visibility such that, at some point during the graph cleaning process
    on $G$, there are at most $r$ vertices with gas.
\end{definition}

In fact, the inference number can only change by as much as the value of the
parameter $r$ has changed.

\begin{proposition}
    For every positive integer $l$, if $G$ is a graph and $r\leqslant s$, then
    $\ded^r_l(G)-\ded^s_l(G)\leqslant s-r$.
\end{proposition}

\begin{proof}
    We will show that $\ded_l^r(G)\le \ded_l^s(G)+s-r$ by showing that
    $\ded_l^s(G)+d$ many cleaners for some $d\le s-r$ can clean all but at most
    $r$ vertices of $G$. So, take any strategy of $\ded_l^s(G)$ cleaners that
    leaves at most $s$ vertices gaseous, and let $v_1,\dots,v_t$ for $0\le t\le
    s$ be all vertices that are left gaseous at the first instance $i_0$ where
    there are at most $s$ gaseous vertices. Note that, in the gas cleaning
    process, there is no adversary, and hence any strategy for the cleaners
    defines a deterministic process; so, this instance $i_0$ is well-defined
    given some optimal strategy for the cleaners.

    Now we alter the strategy by placing an additional $d=\min(s-r,\,t)$
    cleaners on vertices $v_1,\dots,v_{d}$ at the beginning of the game and
    leaving them stationary in all subsequent moves, while the remaining
    $\ded_l^s(G)$ cleaners move following the original strategy. Note that any
    vertex that was clean at some point in the original strategy will be clean
    at the same point following the altered strategy, since the additional
    cleaners can only hamper the spreading of the gas. But we have at least $d$
    additional vertices cleaned at instance $i_0$: these are simply the
    vertices $v_1,\dots,v_d$, where we placed the additional cleaners. In
    total, we found a strategy with $\ded_l^s(G)+d$ cleaners that cleans all
    but at most $t-d\le r$ vertices, which yields the result.
\end{proof}

Note that we will not actually always need the extra $s-r$ cleaners, and we
could gain huge improvements depending on what the exact configuration of the
$m$ gaseous vertices is at that timestep: if the $s$ vertices are a clique,
then we need only one extra cleaner (more generally, it would suffice to
dominate an $(s-r)$-subset of $s$-set, for example).

\begin{corollary}
    If $G$ is a graph, then $\cop'_l(G)-\ded^1_l(G)\leqslant   1$ for all
    positive integers $l$.
\end{corollary}

Observe that $\ded^1_l(G)$ is the minimum number of cleaners such that we can
be certain of the robber's location at some point: the cops might not
\emph{see} the robber, but they have narrowed down the location to exactly one
vertex, and hence the robber must either be there, or else they must have been
seen before. So, the corollary above shows that the difference between wanting
to \emph{see} the robber versus wanting to be certain of the robber's location
is worth at most 1 extra cop.

It might not be immediately obvious that there exists a graph where we know
where the robber is without being able to see them---i.e.\ that there exists a
graph $G$ with $\cop'_l(G)-\ded^1_l(G)=1$---but indeed there does! Our graph
$C_5$ from \cref{fig:C5} illustrates that $\cop'_1(C_5)=2$ but
$\ded^1_1(C_5)=1$. We discuss this again briefly in \cref{sec:prob} during a
probabilistic digestif.

The graph $C_5$ appears to be more than just a simple example illustrating the
difference between being able to see the robber and knowing their location.
Indeed, if the reader starts to investigate the following open problem
(particularly with $l=1$), then they will see it cropping up frequently, but it
is not clear exactly what is going on.

\begin{problem}
    Characterise the graphs $G$ for which $\cop'_l(G)=\ded^1_l(G)$.
\end{problem}

\subsection{Cleaning on a budget}

If we do not have enough cops to clean a graph completely, one might still ask
the following question: what is the best we can do? That is, how \emph{close}
can we get to completely cleaning the graph? One could imagine that this is a
much more practical scenario, where perhaps there is not enough budget to hire
a sufficient number of cops, and so we are interested in how we can best make
do with what we have. This might be considered a dual problem to the inference
number $\ded^r_l$ in the last subsection (\cref{def:inference}).

We will give a few basic results on how well we can clean with general
parameters, and then we focus on having two 1-visibility cops.

The simplest case is when we have just a single $l$-visibility cop. Recall that
$\Delta(G)$ is used to denote the maximum degree of a graph $G$. That is,
$\Delta(G)=\max_{v\in G}(|N(v)|)$, where $N(v)$ is the set of neighbours of
$v$. We write $N_k(v)$ for the set of $k$-neighbours of $v$: the set of
vertices (distinct from $v$) of distance at most $k$ from $v$. So,
$N(v)=N_1(v)$. We then define $\Delta_k(G)=\max_{v\in G}(|N_k(v)|)$.

\begin{proposition}
    If $G$ is a graph, then one $l$-visibility cop can always clean at least
    $\Delta_l(G)+1$ vertices. Additionally, if $G$ is connected, then they can
    clean either the whole graph or else at least $\Delta_l(G)+2$ vertices.
\end{proposition}

\begin{proof}
    We place our cop on any vertex $v$ with $|N_l(v)|=\Delta_l(G)$, thus
    already cleaning $\Delta(G)+1$ vertices.

    Suppose further that $G$ is connected. Either we have cleaned the whole
    graph already, or else there must be some vertex $u$ outside of the
    $l$-neighbourhood of $v$. Thus, moving our cop one step towards $u$ will
    clean at least one additional vertex, giving the result.
\end{proof}

Note that this is the best we can hope for in general since, on any cycle, a
single $l$-visibility cop can clean either the whole graph (if the cycle is
sufficiently small), or else at most $2l+2=\Delta_l(C_n)+2$ vertices.

If we grant ourselves $k$ $l$-visibility cops, and our graph has sufficiently
many disjoint (connected) subgraphs $G_1,\dots,G_k$ (which are themselves
sufficiently large), then we can clean at least
$\sum_{i=1}^k\Delta_l(G)_i+2\cdot k$.

\begin{proposition}
    \label{prop:general-sub}
    If $G$ is a graph with at least $k$ disjoint, connected subgraphs
    $G_1,\dots,G_n$ where each $G_i$ has at least $\Delta_l(G_i)+2$ vertices,
    then $k$ $l$-visibility cops can clean at least
    $\sum_{i=1}^k\Delta_l(G_i)+2\cdot k$ (where $\Delta_l(G_i)$ is calculated
    using only vertices in $G_i$, rather than $G$ as a whole).
\end{proposition}

\begin{proof}
    Place each cop on a vertex $v_i$ in $G_i$ where $|N_l(v_i)|=\Delta_l(G_i)$
    (restricted to vertices only in $G_i$). Since the $G_i$ are disjoint, the
    cops have cleaned at least $\sum_{i=1}^k\Delta_l(G_i)+k$ vertices. Since
    each $G_i$ is connected and has at least $\Delta_l(G_i)+2$ vertices, there
    exists a vertex $u_i$ in each $G_i$ that is not in $|N_l(v_i)|$.
    Simultaneously, each cop will move one step on a path towards their
    respective $u_i$, thus ensuring that at least an additional $k$ vertices
    are cleaned, which makes a total of $\sum_{i=1}^k\Delta_l(G_i)+2\cdot k$.
\end{proof}

Of course, dropping the `connected' hypothesis, or that subgraphs need be
sufficiently large, will just reduce our lower bound to $\Delta_l(G_i)+k$.

As a simple corollary, to preface our main result for this section, consider
when a graph has at least $k$ disjoint $P_{2l+2}$ subgraphs.

\begin{proposition}
    If $G$ is a graph with at least $k$ disjoint $P_{2l+2}$ subgraphs, then $k$
    $l$-visibility cops can clean at least $k\cdot(2l+2)$ vertices.
\end{proposition}

\begin{proof}
    Clearly $\Delta_l(P_{2l+2})=2l$, and hence \cref{prop:general-sub} yields
    that our cops can clean at least $2l\cdot k + 2\cdot k=k\cdot(2l+2)$.
\end{proof}

In particular, if we have a graph $G$ with at least 2 disjoint $P_4$ subgraphs,
then two 1-visibility cops can clean at least 8 vertices. We will show now
that, for a connected graph, two 1-visibility cops can \emph{always} clean at
least 8 vertices, if not the whole graph. We will conjecture later
(\cref{conj:10}) that we can always clean at least 10 vertices, and so in some
cases of the proof we will show that we can indeed clean 10.

\begin{theorem}
    \label{thm:>=8}
    If $G$ is a connected graph, then two $1$-visibility cops can always clean
    either at least $8$ vertices of $G$, or else the whole graph.
\end{theorem}

\begin{proof}
    We proceed via case analysis.
    \begin{enumerate}
        \item {\bf $G$ has at most one vertex $v$ with $d(v)\geq3$.}

            In this case, $G$ is either a spider, a cycle, or the union of
            cycles and paths that are connected at a single vertex.

        \item {\bf Any two vertices $u,v$ with $d(u)\geqslant 3$ and $d(v)
            \geqslant 3$ are at distance~$1$.}

            In this case, $G$ consists of a clique $K_k$ and disjoint cycles
            and paths which are connected to this $K_k$ (see \cref{fig:case2}).
            Again, it is easy to check that two cops can clean the entirety of
            $G$ in this case: one cop can completely clean the $K_k$ by being
            placed on any of its vertices, while at the same time preventing
            the gas from spreading from any vertex outside of the clique to the
            clique; while this cop remains stationary, the other cop can
            successively clean each of the disjoint paths and cycles.

            \begin{figure}
                \begin{center}
                    \includegraphics[width=7cm,
                    page=2]{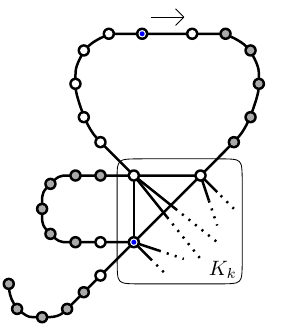}
                    \caption{
                        All vertices with degree at least 3 are at distance 1
                        from each other, and hence form a clique.
                    }
                    \label{fig:case2}
                \end{center}
            \end{figure}

\item {\bf $G$ has at least two vertices $u,v$ with $d(u)\geqslant 3$ and $d(v)
    \geq 3$ which are at distance at least $3$ from each other.}

We will start by placing both cops on $u$ and $v$. Either $u,v$, or any of its
neighbours will have other neighbours in the graph, since the graph is
connected. If $u$ and $v$ are of degree exactly $3$, we can either find a
neighbour of $u$ and a neighbour of $v$ such that both have a neighbour that is
not adjacent to any of the cleaned vertices (see~\cref{fig:case3}), or a)
without loss of generality, all neighbours of $u$ will only be connected to $u$
and $N(v)$ or b) our graph will only consist of 9 vertices. For a) we can
recycle the same argument from the previous case to clean the whole graph (or
our choice of $u$ and $v$ was suboptimal), and for b) it is quite easy to check
that two cops can clean the whole graph as well. Otherwise, we can clean at
least $10$ vertices in this case.

\begin{figure}
\begin{center}
    \includegraphics[width=9cm,page=4]{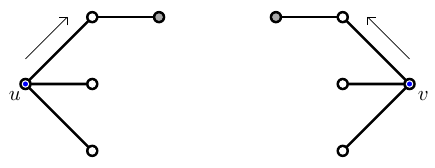}
        \caption{Two neighbours of vertices of degree at least 3 each have at
        least one further neighbour.}
        \label{fig:case3}
    \end{center}
\end{figure}

\item {\bf $G$ has at least two vertices $u,v$ with $d(u)\geqslant 3$ and $d(v)
    \geq 3$ that are at distance $2$, and any two vertices of degree at least
$3$ are at distance $2$ or less from each other.}

Here we have to differentiate between a few subcases. First, let us consider
the case where we have exactly two vertices $u,v$ with degree at least $3$ that
are at distance $2$. Then all other vertices with degree at least $3$ form a
clique (or there are no further vertices of degree at least 3) and are
connected to both $u$ and $v$, and we are in a similar case as before where all
vertices of degree at least $3$ are at distance $1$ to each other. We can place
one cop on this clique (or any common neighbour of $u$ and $v$ if there exists
no clique) and have the other cop successively clean the rest of the graph,
thus cleaning the whole graph (see \cref{fig:lastpage}).

\begin{figure}
\begin{center}
    \includegraphics[width=9cm,page=9]{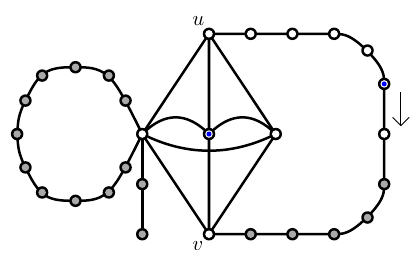}
        \caption{All neighbours of $u$ and $v$ with degree at least $3$ form a
        clique.}
        \label{fig:lastpage}
    \end{center}
\end{figure}

If $u$ and $v$ have exactly one common neighbour, we can place both cops on $u$
and $v$, respectively. If, without loss of generality, $d(v) \geqslant 4$, then
we will immediately clean at least $8$ vertices. Otherwise, one of the
neighbours of $u$ and $v$ will have another neighbour that is not in $N(u) \cup
N(v)$ (or we already cleaned the whole graph), thus moving one cop to this
vertex will yield $8$ cleaned vertices (see~\cref{fig:case4}).

\begin{figure}
\centering
\includegraphics[width=9cm,page=5]{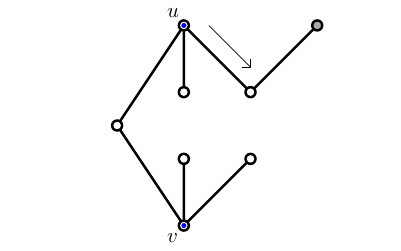}
        \caption{Two vertices with one common neighbour.}
        \label{fig:case4}
\end{figure}

If $u$ and $v$ have exactly two common neighbours, if $G$ has $6$ vertices, we
can clean the whole graph, by placing both cops on $u$ and $v$, respectively.
If $G$ has $7$ vertices, we can repeat the argument from before to clean at
least $7$ vertices (thus the whole graph). If $G$ has at least $8$ vertices,
since $G$ is connected, we can briefly consider the configurations, in which
these vertices can be connected to our graph.

\begin{itemize}
    \item If at least two vertices (say $w_1, w_2$) are incident to $\{u\} \cup
        \{v\} \cup N(u) \cup N(v)$, we can place our cops on $u$ and $v$ and
        then in the next turn move those to vertices that are incident to $w_1$
        and $w_2$ (in case these vertices are not incident to $u$ and $v$),
        thus cleaning at least $8$ vertices.

    \item If only one vertex (say $w_1$) is incident to $\{u\} \cup \{v\} \cup
        N(u) \cup N(v)$, there has to be at least one more vertex which is
        incident to $w_1$ (see \cref{fig:case5}). Thus we can clean at least
        $8$ vertices by placing a cop on $w_1$ in $N(u) \cup N(v)$ and moving
        it to the outside neighbour of $w_1$ in the next turn, while the other
        cop stays stationary at $v$ (or $u$) or can be used to clean the rest
        of $\{u\} \cup \{v\} \cup N(u) \cup N(v)$.
\begin{figure}
\centering
\includegraphics[width=9cm,page=6]{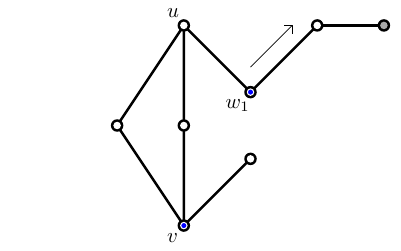}
        \caption{Two vertices with two common neighbours.}
        \label{fig:case5}
\end{figure}

\end{itemize}

If $u$ and $v$ have at least three common neighbours, let us again consider a
few subcases separately.

\begin{itemize}
    \item If at least $3$ further vertices are connected to $u,v$, and/or at
        most $2$ vertices in $N(u) \cup N(v)$, we can place the cops on $u$ and
        $v$ and then move them to the vertices which are incident to the $3$
        other vertices (if necessary), thus cleaning at least $8$ vertices.
    \item If there are at most $2$ further vertices incident to $\{u\} \cup
        \{v\} \cup N(u) \cup N(v)$, we can place the cops on $u$ and $v$ and
        then move them to the vertices incident to outgoing edges in at most
        one turn. Then we have either cleaned the whole graph, or one of the
        outer vertices has to have at least one more neighbour, thus by moving
        a cop to this outer vertex we clean at least $8$ vertices.
    \item The last case to consider is the case where there are at least $3$
        outgoing edges to at least $3$ further vertices, none of which are
        connected to $u$ or $v$ (see~\cref{fig:case6}). We have to consider two
        subcases.

        \begin{figure}[h]
            \begin{center}
                \includegraphics[width=8cm,page=7]{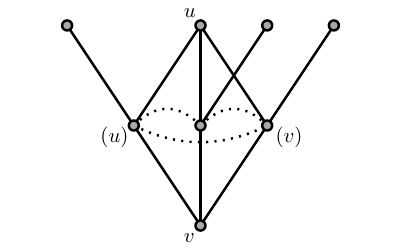}
                \caption{Two vertices with at least 3 common neighbours.}
                \label{fig:case6}
            \end{center}
        \end{figure}

        If every edge in $N(u) \cap N(v)$ is present in the graph, we can place
        our cops on two vertices of $N(u) \cap N(v)$ with outgoing edges, and
        then move one of the cops to another vertex in $N(u) \cap N(v)$ with an
        outgoing edge, thus cleaning at least $8$ vertices. If at least one of
        the edges in $N(u) \cap N(v)$ is missing, we can chose two non-adjacent
        vertices in $N(u) \cap N(v)$, which we can then consider to be $u$ and
        $v$ in our previous case, unless both vertices have another neighbour
        in common, in which case we are in the first subcase of this case. In
        both cases, we can thus clean at least $8$ vertices.

\end{itemize}

\end{enumerate}

\end{proof}

A budget manager looking to extract maximum utility from their cops would be
curious about the following question: given $k$ $l$-visibility cops, what is
the smallest (connected) graph that they can't clean? For two 1-visibility
cops, we have shown that we can clean all connected graphs on 8 vertices or
fewer. We have further verified computationally that in fact they can clean
every connected graph on up to 10 vertices. As we will prove in the next
section, the Heawood graph on 14 vertices (see \cref{fig:heawood}) cannot be
cleaned by two 1-visibility cops, but our computations indicate that they can
clean at most 10 of its vertices. This leads us to conjecture that 10 cleaned
vertices is the best one can hope to achieve.

\begin{conjecture}
    \label{conj:10}
    If $G$ is a connected graph, then two 1-visibility cops can clean at least
    10 vertices, if not the whole graph.
\end{conjecture}

\begin{figure}
    \begin{center}
        \includegraphics[width=7cm,page=8]{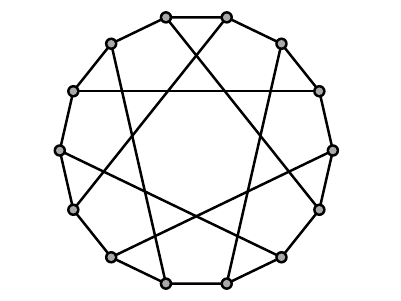}
        \caption{The Heawood graph}
        \label{fig:heawood}
    \end{center}
\end{figure}

\subsection{Cages are tough to clean}
\label{subsec:girth}

In this section, we investigate the plight of every zookeeper: cleaning cages.
In particular, we give upper bounds on the smallest graphs that cannot be
cleaned by $k$ $l$-visibility cops; we show that, if $G$ is a graph with
$\delta(G)\geqslant k$ and girth at least $2l+4$, then we need at least $k$
$l$-visibility cops to clean it. The idea is similar to the one used in
\cite{aigner.fromme:game} to show that the cop number of a graph with girth
greater than $4$ is at least its minimal degree.

\begin{theorem}
    \label{thm:girth}
    If $G$ is a graph with $\delta(G)\geqslant k$ and girth at least $2l+4$,
    then $\cop'_l(G)\geqslant k$.
\end{theorem}

\begin{proof}
    Let $l$ be given and let $v$ be a vertex of $G$ that is gaseous. We will
    show that in the next step there will always exist a vertex $u \in N(v)$
    that will be gaseous. To show this, we will show that each cop can prevent
    the gas from spreading to at most one neighbour of $v$, thus showing that
    $k$ of them are needed. We will say that an $l$-visibility cop $c$
    \textit{blocks} vertex $u \in N(G)$ if it is at a distance of at most $l$
    from $c$ or if it is at a distance of at most $l$ after the cop has moved
    once (after the spreading of the gas). In other words, we won't say that a
    vertex is gaseous if the cop will clean it in the very next move. Assume
    now the contrary, that there is a cop $c$ that can block at least two
    neighbours of $v$, say $u_1$ and $u_2$. If they are blocked in the first
    described sense---if they are in the $l$-neighbourhood of $c$ prior to the
    spreading of the gas---then we have that $v, u_1, c, u_2$ lie on a
    $(2l+2)$-cycle, which contradicts our assumption; see \cref{fig:girth} for
    the case $l=1.$ If one of them is in the $l$-neighbourhood of $c$ prior to
    the spreading of the gas, say $u_1,$ and the other is seen by the cop after
    he has moved, say $u_2$, that would mean that there is a vertex $c'$ which
    is a direct neighbour of $c$ and in whose $l$-neighbourhood $u_2$ is. But
    then $v, u_1 , c, c', u_2$ would lie on a $(2l+3 )$-cycle, which is again
    in contradiction with our assumptions. That the cop cannot block two
    vertices after the gas spreads is similar to the first case, since we would
    again have a $(2l+2)$-cycle.

    Therefore, each cop can block at most one neighbour of $v$ and since there
    are at least $k$ neighbours for every vertex $v \in G$, at each step the
    gas will spread to another vertex with the same property. As we start with
    at least one gaseous vertex, we conclude that $k$ cops cannot clean the
    whole graph.
\end{proof}

\begin{figure}[t]
    \centering

    \begin{subfigure}[b]{0.45\textwidth}
        \centering
        \includegraphics[ page =1]{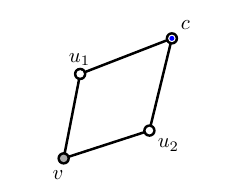}
        \subcaption{
            The first described case from the theorem: the cop $c$ sees both
            $u_1$ and $u_2$ before the spreading of the gas.
            \phantom{texttexttext}
        }
        \label{fig:girh-1}
    \end{subfigure}
    \hfill
    \begin{subfigure}[b]{0.45\textwidth}
        \centering
        \includegraphics[page =2]{girth.pdf}
        \subcaption{
            The second described case from the theorem: the cop $c$ sees $u_1$
            prior to the spreading of the gas and $u_2$ after he has moved to
            $c'$.
        }
        \label{fig:girth-2}
    \end{subfigure}
    \caption{
        Illustration of the cases from the proof of Theorem \ref{thm:girth} for
        $l=1.$
    }
    \label{fig:girth}
\end{figure}

Graphs of girth at least $m$ and minimum degree at least $n$ exist whenever
$n\geqslant 2 $ and $m\geqslant 3$, and they are well studied
\cite{exoo.jajcay:dynamic}; and those with a minimum number of vertices (for a
pair $(n,m)$) are called \emph{$(n,m)$-cages}). In particular, if the
visibility is $l=1$, then there exist $k$-regular graphs of girth $6$ with no
more than $2(1+\varepsilon)k^2$ vertices, for every $\varepsilon>0$ and $k$
large enough~\cite{brown:on} (they can be constructed by starting from the
incidence graph of a projective plane of order $\geqslant k-1$ and then erasing
some conveniently chosen vertices and edges). This result is the best possible,
as it is known that any such graph must have at least $2k^2$ vertices. (Note,
however, that we do not disallow the possibility that some completely different
construction than the one from Theorem \ref{thm:girth} might give a graph with
less then $2k^2$ vertices not cleanable by $k$ cleaners with visibility $1$;
we, however, strongly doubt that this is true.) For $l>1$, that is, if a girth
of at least 8 is fixed, then only some messier upper bounds for the minimum
number of vertices with respect to $\delta(G)$ are known (see
again~\cite{exoo.jajcay:dynamic}).

The Heawood graph we mentioned at the end of the last section (see
\cref{fig:heawood} is a (3,6)-cage with 14 vertices. Thus, by \cref{thm:girth},
we know that we need at least three 1-visibility cops to clean it. It would be
interesting to consider, in general, a smallest graph that cannot be cleaned by
$k$ $l$-visibility cops. We proved earlier in \cref{thm:>=8} that two
1-visibility cops can clean all graphs on at most 8 vertices (and we verified
computationally that they can clean all graphs on at most 10 vertices). So, a
smallest graph that two 1-visibility cops cannot clean must have between 11 and
14 vertices inclusive.

\begin{problem}
    What is a smallest graph that two 1-visibility cops cannot clean? More
    generally, what about for $k$ $l$-visibility cops?
\end{problem}

\section{Stumbling across the robber by chance}
\label{sec:prob}

As discussed before, cleaning a graph or trying to find a robber with limited
visibility are two sides of the same coin. But, if we apply those processes in
practice, then there are noticeable differences between the two philosophies
when chance becomes a factor, which we will explain with the following example.
Essentially, we wish to discuss what happens when we forgo the omniscience of
the robber.

We mentioned earlier that $\cop_1'(C_5)=2$ and $\ded_1^1(C_5)=1$ (see
\cref{fig:C5}). To be more explicit about why this is the case, consider the
following: on any cycle, two cops are enough to catch the robber (even with
$0$-visibility) by simply moving around in opposite directions, so
$\cop_1'(C_5)\le 2$. But one cop is clearly not enough: in terms of the gas
cleaning process, at the beginning of the cop's turn there are exactly three
cleaned vertices, and in whichever direction the cop chooses to move, they will
clean only one additional vertex, still leaving one vertex gaseous. Therefore,
$\cop_1'(C_5)=2$ and $\ded_1^1(C_5)=1$.

One cop is, however, much closer to seeing the robber than it might seem at
first: the robber is always at distance at most 2 from the cop, so if at any
point in the game the cop chose the direction of the shortest path to the
robber, they would reduce the distance to 1 and see the robber. Although the
cop cannot know what this shortest path to the robber is, if they would choose
their direction by simply flipping a coin, they would have $50\%$ chance of
winning at any point in the game. So for any $n$, the probability of the cop
failing to see the robber after $n$ turns is $1/2^n$, which drops to $0$ in the
limit if they play indefinitely. That is, the probability of eventually seeing
the robber is $1$. This leads us to the following definition.

\begin{definition}
    Given a graph $G$ and a non-negative integer $l$, we define the
    \emph{eventual capture number} with $l$-visibility cops, written
    $\evcapt_l(G)$, as the minimum number of cops required such that the
    probability of eventually capturing the robber is 1, where at each turn
    cops choose their move uniformly from the set of all possible moves.

    Similarly, we define the \emph{eventual seeing number} with $l$-visibility
    cops, written $\evcapt'_l(G)$, as the minimum number of cops required such
    that the probability of eventually seeing the robber is 1 in the same
    probability space as above.
\end{definition}

Note that asking for probability 1 in the above definition is the same as
asking for non-zero probability (for finite graphs). In fact, it is even
possible to represent these two numbers deterministically.

\begin{proposition}
    \label{prop:ec=c}
    If $G$ is a graph, then $\evcapt_l(G)=\cop(G)$ for all non-negative
    integers $l$. Moreover, for all $l$, $\evcapt'_l(G)$ is equal to the
    minimum number $\reach_l(G)$ of classical cops (i.e.\ with non-restricted
    visibility) that have a strategy to reach a vertex in $G$ that is at
    distance at most $l$ from the robber. 
\end{proposition}

\begin{proof}
    First note that $\evcapt_l(G)=\evcapt'_0(G)$. This is because seeing the
    robber with $0$-visibility is the same as capturing them, and the strategy
    for the eventual capture is chosen at random, i.e. irrespective of the
    value of $l$. Therefore it suffices to show the second part of the
    statement.

    Now suppose $\reach_l(G)=k$. That is, $k$ cops have a strategy to get
    within distance at most $l$ from the robber. But then with probability $1$
    the $l$-visibility cops will play this optimal strategy for long enough to
    guarantee seeing the robber, since the graph is finite and there is only a
    limited number of possible moves available at each turn. This shows
    $\evcapt'_l(G)\le \reach_l(G)$.

    For the other direction, suppose that at most $k-1$ $l$-visibility cops
    play the probabilistic game. By definition, the robber has a strategy to
    avoid at most $k-1$ standard cops indefinitely, regardless of their moves.
    Consequently, the robber, whose vision is not restricted, can avoid at most
    $k-1$ cops moving randomly, as well. This concludes the proof.
\end{proof}

Note that the number $\reach_l(G)$ from the previous proposition has already
been studied by some researchers, e.g.\ in \cite{bonato.chiniforooshan:pursuit,
chalopin.chepoi.ea:cop}.

So, it is now clear that we have the following relationships:
\begin{align*}
    \reach_l(G)=\evcapt'_l(G)&\leq\evcapt_l(G)=\cop(G)\leq\cop_l(G)\\
    \cop'_l(G)&\leq\evcapt_l(G)=\cop(G)\leq\cop_l(G).
\end{align*}

A natural question to ask is whether the difference $\cop'_l(G)-\evcapt'_l(G)$
can be arbitrarily large, and indeed it can. In \cite[Theorem 3.1 and Lemma
4.4]{clarke.cox.ea:limited}, it is shown that, for every positive integer $k$,
there exists a tree $T$ with $\cop'_l(T)=k+1$. On the other hand, $\cop(T)=1$
in any tree, which gives $\evcapt'_l(T)=\reach_l(T)\le\cop(T)=1$, and hence
$\cop'_l(T)-\evcapt'_l(T)=k$.

Note that $\cop(G)=\reach_0(G)$ for all $G$. So, in the example above, we had
$\reach_l(G)\le\reach_0(G)=\cop(G)<\cop_l'(G)\le\cop_l(G)$. When additionally
$l\ge 2$, this implies by \cite[Theorem 3.1]{clarke.cox.ea:limited} that
$\cop_l(G)=\cop_l'(G)$. On the other hand, the chain of inequalities might as
well read $\reach_l(G)\le\cop_l'(G)\le\cop(G)\le\cop_l(G)$ (with the additional
inequality $\cop_l(G)\le\cop(G)+1$ if $l\ge2$, again by \cite[Theorem
3.1]{clarke.cox.ea:limited}). When this is the case, it is perfectly possible
for the first inequality to be an equality (e.g. for $l=1$ and $G=C_4$, we have
$\reach_1(C_4)=\cop_1'(C_4)=1$). Intuitively, $\reach_l(G)=\cop_l'(G)$ means
that lacking knowledge of the robber's position does not add any extra
difficulty to seeing them: the ``hardest'' part is actually reaching a position
from where they can possibly be seen. Therefore, it might be interesting to
characterise the graphs where this is the case.

\begin{problem}
    Characterise the graphs $G$ for which $\reach_l(G)=\cop'_l(G)$.
\end{problem}

Another natural question, given \cref{prop:ec=c}, is the following: even though
$\evcapt_l(G)=\cop(G)$, the expected amount of time it takes for the
$l$-visibility cops to win will not always be equal to the amount of time it
takes for the classical cops to capture the robber. The study of capture time
for classical cops and robbers was introduced in
\cite{bonato.golovach.ea:time}, and has received considerable interest since.
Perhaps it would be interesting to consider the difference between the expected
$l$-visibility capture time and the classical capture time. For example, using
two cops, the classical capture time in $C_5$ is 1, but the expected capture
time with 1-visibility is 2.

\section{Acknowledgements}

The authors are grateful for the Games and Graphs Workshop, held 28th--31st of
October 2024 in Lyon, France, which birthed this collaboration; in particular,
thanks are due to the organising committee for their hospitality and
organisation of the workshop: Eric Duch\^ene, Arthur Dumas, Valentin Gledel,
Fionn Mc Inerney, Aline Parreau, and Th\'eo Pierron.

\bibliographystyle{plainurl}
\bibliography{references}

\end{document}